\documentclass[twoside]{siamltex}
\usepackage{amsmath,amssymb}
\usepackage[margin=1in]{geometry}
\usepackage[noadjust]{cite}

\def\cvd{~\vbox{\hrule\hbox{%
     \vrule height1.3ex\hskip0.8ex\vrule}\hrule } }

\usepackage{tikz}
\usepackage{caption}

\newcommand{\join}{\vee}
\DeclareMathOperator{\diam}{diam}
\DeclareMathOperator{\cc}{cc}
\DeclareMathOperator{\mr}{mr}
\DeclareMathOperator{\M}{M}

\begin{document}

%  Leave these commented lines here 
% \input{elaheader-volx-xx.tex}
% \setcounter{page}{1}

% \renewcommand{\thefootnote}{\fnsymbol{footnote}}
% \renewcommand{\thefootnote}{\arabic{footnote}}
% \renewcommand{\theequation}{\thesection.\arabic{equation}}

\bibliographystyle{plain}
\title{Minimum number of distinct eigenvalues of graphs
\thanks{Received by the editors on Month x, 200x.
Accepted for publication on Month y, 200y   Handling Editor: .}}
% Leave blank; editors will write the exact dates above

\author{
Bahman Ahmadi\thanks{Department of Mathematics and Statistics,
University of Regina, Regina, Saskatchewan, S4S 0A2.
(bahman.ahmadi@gmail.com, fatemeh.naghipour@gmail.com). }
% Remember to put \and between any two authors
\and
Fatemeh Alinaghipour\footnotemark[2]
\and
Michael S. Cavers
\thanks{Department of Mathematics and Statistics, University of Calgary,
Calgary, AB, T2N 1N4. (michael.cavers@ucalgary.ca)}
\and
Shaun Fallat\thanks{Department of Mathematics and Statistics,
University of Regina, Regina, Saskatchewan, S4S 0A2.Research supported in part by an NSERC research grant. (shaun.fallat@uregina.ca)} 
\and
Karen Meagher\thanks{Department of Mathematics and Statistics,
University of Regina, Regina, Saskatchewan, S4S 0A2. Research supported in part by an NSERC research grant. (karen.meagher@uregina.ca)}
\and 
Shahla Nasserasr\thanks{Department of Mathematics and Statistics,
University of Regina, Regina, Saskatchewan, S4S 0A2. Research supported by PIMS and Fallat's NSERC Research Grant. (Shahla.Nasserasr@uregina.ca)}
}
% Note that \footnotemark[3]} is used for the third author
% because of the same affiliation for the second and third authors.
% If the same affiliation is to be used for the first and second authors,
% \footnotemark[2] should be used instead of \thanks{} for the second author.

% Authors and running title to go on top of each page
\pagestyle{myheadings}
\markboth{B. Ahmadi, F. Alinaghipour, M.S. Cavers, S. Fallat, K. Meagher, S. Nasserasr}{Minimum Number of Eigenvalues}
\maketitle

\begin{abstract}
The minimum number of distinct eigenvalues, taken over all real symmetric matrices compatible
with a given graph $G$, is denoted by $q(G)$. Using other parameters related to $G$, bounds for $q(G)$ are proven and then applied to deduce further properties of $q(G)$. It is shown that there is a great number of graphs $G$ for which $q(G)=2$. For some families of graphs, such as the join of a graph with itself, complete bipartite graphs, and cycles, this minimum value is obtained. Moreover, examples of graphs $G$ are provided to show that adding and deleting edges or vertices can dramatically change the value of $q(G)$. Finally, the set of graphs $G$ with $q(G)$ near the number of vertices is shown to be a subset of known families of graphs with small maximum multiplicity.
\end{abstract}


%%%%%%%%%%%%%%%%%%%%%%%%%%%%%%%%%%%%%%%%%%%%%%%%%%%%%%%%%%%%%

\section{Introduction}

Suppose $G=(V,E)$ is a simple graph with vertex set
$V=\{1,2,\ldots,n\}$ and edge set $E$. To a graph $G$, we associate
the collection of real $n \times n$ symmetric matrices defined by
\[ 
S(G) = \{ A : A=A^{T}, \; {\rm for} \; i \neq j, \;  a_{ij} \neq 0 \Leftrightarrow \{i,j\} \in E\}.
\]
Note that, the main diagonal entries of $A$ in $S(G)$ are free to be
chosen. %Any matrix in $S(G)$ is said to be \textsl{compatible} with $G$.

For a square matrix $A$, we let $q(A)$ denote the number of
distinct eigenvalues of $A$. For a graph $G$, we define
\[ q(G) = \min \{ q(A) : A \in S(G) \}.\] It is clear that for any
graph $G$ on $n$ vertices, $1 \leq q(G) \leq n$. Furthermore, it is not
difficult to show that for a fixed $n$, there exists a graph $G$ on $n$
vertices with $q(G)=k$, for each $k=1,2,\ldots,n$, see Corollary~\ref{cor:anynk}
for further details.

The class of matrices $S(G)$ has been of interest to many researchers
recently (see \cite{FH, FH2} and the references therein), and there
has been considerable development of the parameters $\M(G)$ (maximum
multiplicity or nullity over $S(G)$) and $\mr(G)$ (minimum rank over
$S(G)$) and their positive semidefinite counterparts, see, for example, the works
\cite{psd, FH, FH2}. Furthermore, as a consequence interest has grown
in connecting these parameters to various combinatorial properties of
$G$.  For example, the inverse eigenvalue problem for graphs (see
\cite{Hog05}) continues to receive considerable and deserved
attention, as it remains one of the most interesting unresolved issues
in combinatorial matrix theory.

In the context of the $(0,1)$-adjacency matrix, $A(G)$, it
is well known that $q(A(G))$ is at least one more than the diameter of $G$
(denoted by ${\rm diam}(G)$) (see \cite{BR}). This result was
generalized to the case of trees, by observing that $q(A) \geq {\rm
  diam}(G) +1$, whenever $A \in S(G)$ is an entry-wise nonnegative
matrix (see \cite{LJ}). Thus if $G$ is a
tree, it is known that $q(G) \geq {\rm diam}(G) +1$. However, it has been
demonstrated that while this inequality is tight for some trees
(e.g. path, star), equality need not hold for all trees (see \cite{BF}
,\cite{KS2}, and also \cite{RUofWy}).

Our main interest lies in studying the value of $q(G)$ for arbitrary
graphs, and as such is a continuation of \cite{F} by de Fonseca. As with many studies of this kind,
moving beyond trees leads to a number of new and interesting
difficulties, and numerous exciting advances. It is clear that
knowledge of $q(G)$ for a graph $G$ will impact current studies of the
inverse eigenvalue problem for graphs, and, in particular, the
parameters $\M(G)$ and $\mr(G)$.

%This parameter has become a vexing fascination for us! 
%Many of the
%results contained in this paper contradict our initial intuition on
%this problem.  
%Here we consider graphs
%$G$ for which $q(G)$ is two. 
Our work has been organized into a number of components. The next section contains
necessary background information and various preliminary-type results
including connections between $q(G)$ and existing graph parameters as
well as the graphs attaining the extreme values of $q(G)$. The following
section provides a simple but surprisingly useful lower bound for $q(G)$.
%for some graphs $G$.  
%
%Next,
%we demonstrate a number of interesting characteristics 
%for  graphs with $q(G)=2$,
%and verify that if $G$ is the join of any connected
%graph with itself, then $q(G)=2$.  This parameter is not 
%monotone on induced 
%subgraphs, in fact, the addition of a single edge or vertex can
%dramatically increase or decrease the minimum number of distinct %eigenvalues.
%There appears to be very few graphs with the minimum number of %%distinct
%eigenvalues equal to one less than the number of vertices, and, %based on 
%some preliminary analysis here, we formulate an idea regarding a 
%characterization of the graphs that satisfy this property.
%
%Our paper is divided into three components. 
%The next section contains
%necessary background information and various preliminary-type 
%results
%including connections between $q(G)$ and existing graph 
%parameters as
%well as the graphs attaining the extreme values of $q(G)$. The 
%next
%section gives a simple but surprisingly useful lower bound for 
%$q(G)$
%for some graphs $G$.  
Section~\ref{sec:q=2} is devoted to studying the
graphs for which $q(G)=2$, which is continued into Section 5, whereas the next section considers bipartite graphs
and certain graph products. The final
two sections focus on the graphs for which $q(G)=|V(G)|-1$ and some possible further work.

\section{Preliminary Results}

To begin we list some basic results about the minimum number of distinct
eigenvalues for graphs. In this work, we use $K_n, K_{m,n}$ and $I_n$ to denote the complete graph on $n$ vertices, complete bipartite graph with parts of sizes $m,n$, and the identity matrix of order $n$, respectively. The notations $M_{m,n}, M_n$ are used for the set of real matrices of order $m\times n$ and $n$, respectively. For $A\in M_n$, the set of eigenvalues of $A$ is denoted by $\sigma(A)$. For graphs $G$ and $H$, $G\cup H$ denotes the graph with vertex set $V(G)\cup V(H)$ and edges $E(G)\cup E(H)$, and is called the {\em union of $G$ and $H$}. 

\begin{lemma}\label{lem:q=1}
For a graph $G$, $ q(G)=1$ if and only if $G$ has no edges. 
\end{lemma}
\begin{proof} If $q(G)=1$, then there is an $A\in
S(G)$ with exactly one eigenvalue, this matrix is a scalar
multiple of the identity matrix, thus the graph $G$ is the empty
graph. Clearly, if $G$ is empty graph, then $ q(G)=1$. \end{proof}

\begin{lemma}
For any $n\geq 2$, we have $q(K_n) = 2$.
\end{lemma}
\begin{proof} The adjacency matrix of $K_n$ has two distinct eigenvalues, so
$q(K_n)\leq 2$ and by Lemma \ref{lem:q=1}
$q(K_n)>1,$ which implies $q(K_n) = 2$.  \end{proof}

\begin{lemma}\label{lem:01}
 If $G$ is a non-empty graph, then for any two distinct real numbers $\mu_1, \mu_2$, there is an $A \in S(G)$ such that
  $q(A) = q(G)$ and $\mu_1,\mu_2\in \sigma(A)$.
\end{lemma}

{\em Proof.} Consider $B \in S(G)$ with $q(B) = q(G)$ and $\lambda_1,\lambda_2\in \sigma(B)$ with $\lambda_1\neq \lambda_2$. Then the matrix $A$ defined below satisfies $q(A) = q(B)$ and $\mu_1,\mu_2\in \sigma(A)$.
\[
\displaystyle A=\frac{\mu_1-\mu_2}{\lambda_1-\lambda_2} \left(B + \displaystyle \frac{\lambda_1\mu_2-\lambda_2\mu_1}{\mu_1-\mu_2}I\right). ~~~~~~ \cvd\]

\begin{corollary}\label{cor:union2}
  If $G$ and $H$ are non-empty graphs with $q(G) = q(H) = 2$, then
  $q(G \cup H)=2$. In particular, if $G$ is the union of non-trivial complete
  graphs, then $q(G) =2$.
\end{corollary}

The parameter $q$ is related to other parameters of graphs, such as 
the minimum rank of the graph.  

\begin{proposition}\label{mrandq}
For any graph $G$, we have $q(G)\leq \mr(G)+1.$
\end{proposition}
\begin{proof} Consider a matrix $A\in S(G)$ with the minimum possible rank,
$\mr(G)$. Then, $A$ has $\mr(G)$ nonzero eigenvalues, so the number of
distinct eigenvalues of $A$ is less than or equal to $\mr(G)+1$. \end{proof}

Clearly, any known upper bound on the minimum rank of a graph can be used as an upper bound for the value
of $q(G)$ for a graph $G$. For example, a {\em clique covering}
of a graph is a collection of complete subgraphs of the graph such
that every edge of the graph is contained in at least one of these
subgraphs.  Then the {\em clique covering number} of a graph is the
fewest number of cliques in a clique covering. This
number is denoted by $\cc(G)$. It is well known that for all graphs
$G$, $\mr(G) \leq \cc(G)$; see \cite{FH}, and thus we have the following corollary.

\begin{corollary}\label{cor:cliquecovering}
Let $G$ be a graph, then $q(G) \leq \cc(G) + 1.$
\end{corollary}

In Corollary~\ref{cor:anynk} a family of graphs is given
for which this bound holds with equality.

We conclude this section with the exact value of $q(C_n)$ where $C_n$
is a cycle on $n$ vertices.  This result can be derived
from~\cite{Ferg}, but in this section we will prove it using work from
\cite{FF}.

\begin{lemma}\label{lem:cycles} Let $C_n$ be the cycle on $n$ vertices. Then 
\[ 
q(C_n) = \left\lceil \frac{n}{2} \right\rceil.
%q(C_n)=\left \{
%\begin{array}{cc}
%\frac{n-1}{2}+1 & \mathrm{if} \ n\ \mathrm{odd},\\
%\frac{n}{2} & \mathrm{if} \ n\ \mathrm{even}.
%\end{array} \right.
\]
\end{lemma}
\begin{proof} First, suppose $n=2k+1$, for some $k\geq 1$. Then, the adjacency
matrix of $C_n$ has exactly $\frac{n-1}{2}+1$ distinct
eigenvalues, these eigenvalues are $2\cos \frac{2\pi j}{n}$, $j=1,\ldots,
n$. On the other hand, using
\cite[Cor. 3.4]{FF}, any eigenvalue of  $A\in
S(C_n)$ has multiplicity at most two, so $q(C_n)\geq \frac{n-1}{2}+1$. Thus, $q(C_n)= \frac{n-1}{2}+1$.

Next, consider $n=2k$, for some $k\geq 2$. Again using \cite[Cor. 3.4]{FF},
$q(C_n)\geq k$. Moreover, by \cite[Thm. 3.3]{FF} for any set of
numbers $\lambda_1=\lambda_{2}> \lambda_3=\lambda_{4}>\ldots >
\lambda_{2k-1}=\lambda_{2k},$ there is an $A\in S(C_{n})$ with
eigenvalues $\lambda_i$, $i=1,\ldots, n$. This implies
that if $n$ is even, then $q(C_n)=\frac{n}{2}.$\end{proof}

Since $\mr(C_n)=n-2$ and $q(C_n)\approx n/2$, we know that for some graphs $G$
there can be a large gap between the parameters $\mr(G)$ and
$q(G)$.

\section{Unique shortest path}

There is only one family of graphs for which the eigenvalues for every
matrix in $S(G)$ are all distinct, these are paths. This
statement is Theorem 3.1 in \cite{F}, and also follows from a result
by Fiedler~\cite{Fiedler}, which states that for a real symmetric matrix $A \in M_n$ and a
diagonal matrix $D$ if $\mathrm{rank}(A + D) \geq n - 1$, then $A\in
S(P_n)$. A path on $n$ vertices is denoted by $P_n$. 

\begin{proposition}\label{paths}
  For a graph $G$, $q(G)=|V(G)|$ if and only if $G$ is a path. 
\end{proposition}

From this we can also conclude that the parameter $q$ is not monotone
on induced subgraphs; as $q(P_n)=n$ while $q(C_n)\approx
n/2.$ The next result is related to a very simple, but often very
effective, lower bound on the minimum number of distinct eigenvalues of a
graph that is based on the length of certain induced paths. Recall that the
{\em length} of a path is simply the number of edges in that path, and that the
{\em distance between two vertices}, (in the same component) is the length of the shortest path between 
those two vertices. 

\begin{theorem}\label{thm:uniqueminpath}
  If there are vertices $u$, $v$ in a connected graph $G$ at distance $d$ and the path of length $d$ from $u$ to $v$ is unique, then $q(G) \geq d+1$.
\end{theorem}
\begin{proof} Assume that $u=v_1,v_2,\dots, v_d, v_{d+1}=v$ is the unique
path of length $d$ from $u$ to $v$. For any $A=[a_{ij}] \in S(G)$, all of the matrices $A, A^2,\ldots, A^{d-1}$
have zero in the position $(u,v)$, while the entry $(u,v)$ of
$A^d$ is equal to $\prod_{i=1}^{d}a_{v_{i} v_{i+1}}\neq 0$. Thus, the matrices $I, A, A^2,
\dots , A^{d}$ are linearly independent and the minimal polynomial of $A$
must have degree at least $d+1$.  \end{proof}

It is important to note that the induced path from $u$ to $v$ in the proof of Theorem~\ref{thm:uniqueminpath} is the
shortest path from $u$ to $v$ and that it is the only path of this
length. The length of such a path is a lower bound on the
diameter of the graph and if the path is not unique, then the bound only
holds for nonnegative matrices.

\begin{corollary}
For any connected graph $G$, if $A \in S(G)$ is nonnegative, then $q(A) \geq \diam(G) +1$.
\end{corollary}

We note that 
Theorem~\ref{thm:uniqueminpath} implies Theorem 3.1 from \cite{F}. 

\begin{theorem}\label{carlos} (\cite[Thm. 3.1]{F})
Suppose $G$ is a connected graph.
  If $P$ is the longest induced path in $G$ for which no edge of $P$ lies on a cycle, then $q(G) \geq |V(P)|$.
\end{theorem}

It is not true that $\diam(G) +1$ is a lower bound
for the minimum number of distinct eigenvalues of an arbitrary graph $G$, see
Corollary~\ref{ex:hypercube} for a counter-example. However, in
the case of trees, since in this case any shortest path between two
vertices is the unique shortest path, we have

\begin{corollary}\label{tree-diam}
For any tree $T$,  $q(T) \geq \diam(T) +1$.
\end{corollary}

There are several other proofs of Corollary \ref{tree-diam}, see \cite{LJ, KS1}.
There are also trees with $q(T)> \diam(T)+1$; see \cite{BF}.  Further, for
any positive integer $d$, there exists a constant $f(d)$ such that for
any tree $T$ with diameter $d$, there is a matrix $A\in S(T)$ with at
most $f(d)$ distinct eigenvalues (this was shown by B. Shader~\cite{BS}
who described $f(d)$ as possibly ``super-super-exponential'').  It has been
shown that $f(d)\geq (9/8)d$ for $d$ large, see \cite{KS1} and \cite{RUofWy}.

Using unique shortest paths, it is possible to construct a connected graph
on $n$ vertices with $q(G) = k$, for any pair of integers $k,n$ with $1\leq k \leq n$.

\begin{corollary}\label{cor:anynk}
For any pair of integers $k,n$ with $1\leq k \leq n$, let $G(n,k)$ be the graph on vertices $v_1,\ldots,v_n$, where vertices $v_1,\ldots,v_{n-k+2}$ form a clique and vertices $v_{n-k+2},v_{n-k+3},\ldots,v_{n-1},v_{n}$ form a path of length $k-2$. Then, $q(G(n,k)) = k$.
\end{corollary}
\begin{proof} There is a unique shortest path of length $k-1$ from $v_{n}$ to any of the vertices $v_1,\ldots,v_{n-k+1}$ and there is a clique covering of the graph consisting of $k-1$ cliques. Thus, by Corollary~\ref{cor:cliquecovering} and Theorem~\ref{thm:uniqueminpath}, $q(G(n,k))=k$.  \end{proof}

\section{Graphs with two distinct eigenvalues}
\label{sec:q=2}

For a graph $G$, $q(G)=2$ means that there is a matrix
$A\in S(G)$ such that $A$ has exactly two distinct eigenvalues, and
there is no matrix in $S(G)$ with only one eigenvalue. Therefore, the
minimal polynomial of $A$ has degree two, thus, $A$ satisfies
$A^2=\alpha A + \beta I$, for some scalars $\alpha$ and
$\beta$. This implies that $A$ and
$A^2$ have exactly the same zero-nonzero pattern on the off-diagonal
entries. Equivalently, for any nonempty graph $G$, $q(G)=2$ if and only if $S(G)$ contains a real symmetric orthogonal matrix $Q$. Using this, we can show the following results with the aid of Theorem \ref{thm:uniqueminpath}. 

\begin{lemma}\label{nopendant}
  If $q(G)=2$, for a connected graph $G$ on $n\geq 3$ vertices, then
  $G$ has no pendant vertex.
\end{lemma}
\begin{proof} Suppose vertex $v_1$ is pendant and suppose its unique neighbor is $v_2$. Since $G$ is connected and has at least $3$
vertices there is another vertex $v_3$ that is adjacent to $v_2$. Thus
there is a unique shortest path from $v_1$ to $v_3$ of length $2$ and
the result follows from Theorem~\ref{thm:uniqueminpath}.\end{proof}

The previous basic result is contained in the next slight generalization by noting that any edge 
incident with a pendant vertex is a {\em cut edge} (that is, its deletion results in a 
disconnected graph).

\begin{lemma}\label{nocutedge}
  Suppose $G$ is a connected graph on $n$ vertices with $n\geq 3$. If
  $q(G)=2$, then there is no cut edge in the graph $G$.
\end{lemma}
\begin{proof} Assume that vertices $v_1$ and $v_2$ form a cut
edge. We can assume without loss of generality that there is another
vertex $v_3$ in $G$ that is adjacent to $v_2$. Thus there is a unique
shortest path from $v_1$ to $v_3$ of length $2$ and the result follows
from Theorem~\ref{thm:uniqueminpath}.  \end{proof}

The next result should be compared to Theorem \ref{carlos}.

\begin{corollary}
  If $G$ is a graph on $n\geq 3$ vertices with $q(G)=2$, then every
  edge in $G$ is contained in a cycle.
\end{corollary}
\begin{proof} Lemmas \ref{nopendant} and \ref{nocutedge} together imply this
result. \end{proof}

Consider $\alpha\subseteq\{1,2,\ldots,m\}$ and $\beta\subseteq\{1,2,\ldots,n\}$. For a matrix $A\in M_{m,n}$, $A[\alpha,\beta]$ denotes the submatrix of $A$ lying in rows indexed by $\alpha$ and columns indexed by $\beta$. Recall that for any vertex $v$ of a graph $G$, the {\em neighborhood
set of $v$}, denoted by $N(v)$, is the set of all vertices in $G$ adjacent to $v$.

\begin{theorem}\label{neighbors}
For a connected graph $G$ on $n$ vertices, if $q(G)=2$, then for any independent set of vertices $\{v_1,v_2,\dots, v_k\}$, we have
\[
\left| \bigcup_{i \neq j} ( N(v_i) \cap N(v_j) ) \right | \geq k.
\]
\end{theorem}
\begin{proof} For the purpose of a contradiction, suppose $G$ is a
graph with $q(G)=2$ and that there exists an independent set $S$ with $|S|=k$ such that
\[ 
\left|  \bigcup_{i\neq j} (N(v_i) \cap N(v_j)) \right| < k, 
\]
and let $X = \bigcup_{i\neq j} (N(v_i) \cap N(v_j))$. Using Lemma~\ref{lem:01}, there exists
a symmetric orthogonal matrix $A\in S(G)$. Consider such $A$ and let $B=A[S,\{1,\ldots,n\}]$. Observe that $B$ is a $k \times n$ matrix and any
column of $B$ not indexed with $X$ contains at most one nonzero entry. Since the rows of $B$ are orthogonal, we deduce that rows
of $C=A[S,X]$ must also be orthogonal. However, $C$ is a $k \times |X|$
matrix with $k>|X|$, and orthogonality of these rows is impossible, as
they are all nonzero. This completes the proof. \end{proof}

The next two statements are immediate, yet interesting, consequences of Theorem \ref{neighbors}.
\begin{corollary}\label{cor:22adj}
  Let $G$ be a connected graph on $n\geq 3$ vertices with $q(G)=2$. Then, any two non-adjacent vertices
  must have at least two common neighbors.
\end{corollary}

\begin{corollary}
  Suppose $q(G)=2$, for a connected graph $G$ on $n\geq 3$ vertices. If
  the vertex $v_1$ has degree exactly two with adjacent vertices $v_2$ and $v_3$, then
every vertex $v$ that is different from $v_2$ and $v_3$,  has exactly the same neighbors as $v_1$.
\end{corollary}

Along these lines, we also note that if $G$ is a connected graph with $q(G)=2$, then for any independent set of vertices $S$, we have $|S|\leq n/2$. Thus for any graph $G$ with $q(G)$ being two, we have a basic upper bound on the size of independent sets in $G$.

As a final example, recall that $q(K_n)=2$ whenever $n \geq 2$. We can build on this result
for complete graphs with a single edge deleted.

\begin{proposition}
Suppose $G$ is obtained from $K_n$ by deleting a single edge $e$. Then 
\[ q(G) = \left\{ \begin{array} {ll}
                   1,  &   {\rm if} \; n=2; \\
                   3, &  {\rm if } \; n=3; \\
                   2,   &  {\rm otherwise}. \end{array}\right.\]
 \end{proposition}
 
\begin{proof} The cases $n=2,3$ follow easily from previous facts. So suppose $n \geq 4$. We will 
construct a symmetric orthogonal matrix $Q$ in $S(G)$, assuming the edge deleted was $e = \{1,n\}$, without loss of generality. In this case set,

\[ u_1 = \frac{1}{\sqrt{n-1}} \left[ \begin{array}{c} 1 \\ \hline e_{n-2} \\ \hline 0 \end{array} \right], \]
where $e_{n-2}$ is the $(n-2)$-vector of all ones. Then choose $u_2'$ to be orthogonal to $u_1$ as 
follows 
\[ u_2' =   \left[ \begin{array}{c} 0 \\ \hline e_{n-3} \\ \hline -(n-3) \\ \hline 1 \end{array} \right], \]                
where $e_{n-3}$ is the $(n-3)$-vector of all ones. Then set $u_2 = \frac{1}{\left\| u_2'\right\|}
u_2'$. Finally, set $Q = I-2(u_1u_1^T + u_2u_2^T)$. Then it follows that $Q$ is orthogonal and a
basic calculation will show that $Q \in S(G)$. Hence $q(G)=2$. \end{proof}
 
\section{Join of two graphs}

In the previous section we found several restrictions on a graph $G$ for which
$q(G)=2$. In this section, we
show that, despite these restrictions, a surprisingly large number of
graphs satisfy this property.

Let $G$ and $H$ be graphs, then the {\em join of $G$ and $H$},
denoted by $G \vee H$, is the graph with vertex set $V(G) \cup V(H)$ and
edge set $E(G) \cup E(H) \cup \{\{g,h\} \;|\; g\in V(G), h\in V(H)\}$.

A real matrix $R$ of order $n$ is called an {\em $M$-matrix} if it can
be written in the form $R=sI-B$ for some $s>0$ and entry-wise nonnegative matrix 
$B$ such that its spectral radius satisfies 
$\rho(B)\leq s$. Recall that the {\em spectral radius} of a square matrix $B$ is 
defined to be $\rho(B) = \max \{ |\lambda| : \lambda \in \sigma(B) \}$. In the case that $\rho(B)<s$, then $R$ is called a
{\em nonsingular $M$-matrix}. Recall that for $A \in M_n$, we call $B \in M_n$ a {\em square root of $A$}
  if $B^2=A$.  In \cite{AS82}, it is shown that an
$M$-matrix $R$ has an $M$-matrix as a square root if and only if $R$
has a certain property (which the authors refer to as {\em property
  c}).   It is also known that all nonsingular $M$-matrices have
``property c''.

The following theorem is proved in \cite{AS82}. If $P$ is a square matrix, $\diag(P)$ means the diagonal entries of $P$. 
\begin{theorem}\cite[Thm. $4$]{AS82}\label{ASthm}
  Let $R$ be an $M$-matrix of order $n$, and let $R=s(I-P)$ be a
  representation of $R$ for sufficiently large $s$ such that
  $\diag(P)$ is entry-wise positive and $\rho(P)\leq 1$.  Then $R$ has an $M$-matrix as a
  square root if and only if $R$ has ``property c.''  In this case,
  let $Y^*$ denote the limit of the sequence generated by
\[
Y_{i+1}=\frac{1}{2}(P+Y_i^2),\quad Y_0=0.
\]
Then $\sqrt s(I-Y^*)$ is an $M$-matrix with ``property c'' which is a square root of $R$.
\end{theorem}

Using Theorem \ref{ASthm}, we can prove the following.

\begin{theorem}\label{join:thm}
Let $G$ be a connected graph, then $q(G\vee G)=2$.
\end{theorem}
\begin{proof}
Suppose $G$ is a connected graph on $n$ vertices.
The goal of this proof is to construct a matrix $P$ such that
\[
Q =\left[ 
\begin{array}{cc}
\sqrt{P} & \sqrt{I-P} \\
\sqrt{I-P} & -\sqrt{P}
\end{array}
\right]
\]
is in $S(G \vee G)$.
If we can construct such a matrix $P$ then $Q^2 = \left[ 
\begin{array}{cc}
I & 0\\
0 & I
\end{array}
\right]$
and $Q$ has exactly two eigenvalues. Let $A(G)$ be the adjacency matrix of $G$ and set
\[
P=\frac{2n-1}{4n^2}\left(\frac{1}{n}A(G)+I\right)^2.
\]
Note that $\diag(P)$ is entry-wise positive.  By Gershgorin's disc theorem (see \cite[pg. 89]{BR}), every
eigenvalue of $\frac{1}{n}A(G)+I$ belongs to the interval $(0,2)$, and
hence, every eigenvalue of $P$ belongs to the interval
$\left(0,\frac{2n-1}{n^2}\right)$.  Therefore, $\rho(P)\leq 1$.

Consider the matrix $R=I-P$. Then $R$ is an $M$-matrix that satisfies
the conditions of Theorem \ref{ASthm} with $s=1$.  Note that if $S$ is
a matrix with eigenvalue $\lambda$, then $1-\lambda$ is an eigenvalue
of $I-S$.  Thus, the eigenvalues of $R$ are in the interval
$\left(\left(\frac{n-1}{n}\right)^2,1\right)$.  Hence, $R$ is
nonsingular and thus has ``property c.''  By Theorem \ref{ASthm}, $R$
has an $M$-matrix as a square root of the form $I-Y^*$, where $Y^*$ is
the limit of the sequence generated by
\begin{eqnarray}
Y_{i+1}=\frac{1}{2}(P+Y_i^2),\quad Y_0=0.\label{ASeqn}
\end{eqnarray}
Note that $Y^*$ satisfies
\begin{eqnarray}
(I-Y^*)^2+\frac{2n-1}{4n^2}\left(\frac{1}{n}A(G)+I\right)^2=I.\label{sq-eqn}
\end{eqnarray}
As $G$ is a connected graph, its adjacency matrix $A(G)$ is an irreducible nonnegative matrix.
Thus, $(aA(G)+bI)^n>0$ for every $a,b>0$, and hence, $Y^*>0$ by (\ref{ASeqn}).

As $P$ is a real symmetric matrix, the sequence (and consequently the
limit) of (\ref{ASeqn}) are real symmetric matrices.  In particular,
$Y^*$ is a real symmetric matrix that may be written as a polynomial
in $A(G)$.  Therefore, $I-Y^*$ commutes with $A(G)$.

If $\lambda_1,\lambda_2,\ldots,\lambda_n$ are the eigenvalues of $Y^*$, then
\[
{\rm trace}(Y^*)=\sum_{i=1}^n \lambda_i<1
\]
as each eigenvalue of $Y^*$ belongs to the interval $\left(0,\frac{1}{n}\right)$.
Therefore, $\diag(I-Y^*)>0$ implying that $I-Y^*$ is an entry-wise nonzero matrix.

Finally consider the block matrix
\[
Q=\left[\begin{array}{cc}
\frac{\sqrt{2n-1}}{2n}\left(\frac{1}{n}A(G)+I\right) & I-Y^*\\
I-Y^* & -\frac{\sqrt{2n-1}}{2n}\left(\frac{1}{n}A(G)+I\right)\\
\end{array}\right].
\]
By (\ref{sq-eqn}), $Q$ is an orthogonal matrix with two distinct eigenvalues.
As $I-Y^*$ is entry-wise nonzero, $Q\in S(G\vee G)$, hence, $q(G\vee G)=2$.
\end{proof}

Recall that for any graph $G=(V,E)$, the graph $\overline{G} = (V, \overline{E})$, is
called the {\em complement of $G$} whenever, $\overline{E} = \{ \{i,j\} | \{i,j\} \not\in E\}$.

\begin{corollary}
There are graphs $G$ for which the gap between $q(G)$ and $q(\overline{G})$ can 
grow without bound as a function of the number of vertices of $G$.
\end{corollary}
\begin{proof} Let $G = \overline{P_n} \join \overline{P_n}$ with $n \geq 4$. Then $q(G) =2$,
while $q(\overline{G}) = q(P_n \cup P_n) = n$. \end{proof}

Also, note Theorem \ref{join:thm} fails to hold for two different graphs. For example, 
using Theorem \ref{neighbors}, we have that $q(P_1 \vee P_4) >2$. It is still unresolved 
whether or not the condition that $G$ be connected is required in Theorem \ref{join:thm}.

\section{Bipartite Graphs and Graph Products}
\label{sec:bipartite}

Let $G$ be a bipartite graph with parts $X$ and $Y$ such that $0<|X|=m\leq
n=|Y|$. Define $\mathcal{B}(G)$ to be the set of all real $m\times n$
matrices $B=[b_{ij}]$ whose rows and columns are indexed by $X$ and $Y$,
respectively, and for which $b_{ij}\neq 0$ if and only if $\{i,j\}\in E(G)$. We have the following:

\begin{theorem}\label{bipartite}
  For any non-empty bipartite graph $G$, if $B\in \mathcal{B}(G)$,
  then $q(G)\leq 2 q(BB^T) +1$.
\end{theorem}
\begin{proof} 
Let $B\in \mathcal{B}(G)$ and consider $A\in \mathcal{S}(G)$ with $A=\left[ \begin{array}{cc}
0& B \\
B^T & 0 \\
\end{array} \right].$
It is well known that $BB^T$ and $B^TB$ have the same nonzero
eigenvalues, so the number of distinct nonzero eigenvalues of $A^2$ is at most
$q(B^TB)$. Moreover, the eigenvalues of $A$ are of the form $\pm \sqrt{\lambda}$, where $\lambda$ is an eigenvalue of $A^2$. Thus, $A$ has at most $2q(B^TB)+1$ distinct eigenvalues.  \end{proof}

If $B$ is square, then $B^TB$ and $BB^T$ have the same eigenvalues, this implies the following corollary.

\begin{corollary}\label{cor:bipartite2q}
  For any non-empty bipartite graph $G$ with equal sized parts, $q(G)\leq 2 q(BB^T)$.
\end{corollary}

\begin{lemma}\label{lem:bipartite2}
  For any non-empty bipartite graph $G$, if there is a matrix $B\in
  \mathcal{B}(G)$ with orthogonal rows and orthogonal columns, then $q(G)=2$.
\end{lemma}
\begin{proof} If $B\in \mathcal{B}(G)$ has orthogonal rows and orthogonal columns, then $B$ is a square matrix.
Consider $A\in \mathcal{S}(G)$ with $ A=\left[ 
\begin{array}{cc}
0& B \\
B^T & 0 \\
\end{array} 
\right].$
Then, $A^2=I$, which implies that $A$ has at most two distinct eigenvalues. Thus,
by Lemma~\ref{lem:q=1}, $q(G)=2$. \end{proof}

\begin{proposition}\label{m=n}
Consider a bipartite graph $G$ with parts $X$ and $Y$. If $q(G)=2$, then
  $|X|=|Y|$ and there exists an orthogonal matrix $B\in \mathcal{B}(G)$.
\end{proposition}
\begin{proof}
Label the vertices of $G$ so that the vertices of $X$ come first. Then, any matrix in $S(G)$ is of the form 
$A=\left[ \begin{array}{cc}
D_1& B \\\\
B^T & D_2 \\
\end{array} \right],$
where $D_1\in M_{|X|}$ and $D_2\in M_{|Y|}$ are diagonal matrices. Since $q(G)=2$, using Lemma~\ref{lem:01}, $A\in S(G)$ can be chosen with eigenvalues $-1,1$, therefore $A^2=I$. On the other hand,
\[ 
A^2=\left[
\begin{array}{cc}
D_1^2+BB^T& D_1B+BD_2 \\
B^T D_1+ D_2 B^T & B^T B +D_2^2 \\
\end{array} \right].
\]
This implies that $B B^T$ and $B^T B$ are diagonal. Therefore
the rows and columns of $B$ are orthogonal, and hence $|X|=|Y|$. \end{proof}

For any $n\geq 1$, there is a real orthogonal $n\times n$ matrix all
of whose entries are nonzero. For $n=1,2$ this is trivial, and for $n>2$, the matrix $B=I-\frac{2}{n} J$ is such an orthogonal matrix.

Using the above example and Lemma~\ref{lem:bipartite2}, we have
the following.

\begin{corollary}\label{K_{m,n}}
For any $m,n$ with $1\leq m\leq n,$ 
\[
q(K_{m,n}) = \left\{
\begin{array}{lr}
2, &  \textrm{if } m=n; \\
3, &  \textrm{if } m<n.
\end{array}
\right.
\]
\end{corollary}
\begin{proof}
If $m=n$, it is enough to normalize the real orthogonal matrix in
the example proceeding this Corollary and use it in Lemma~\ref{lem:bipartite2}. If $m<n$,
then according to Proposition~\ref{m=n}, we have $q(G)\geq 3$. On the
other hand, the adjacency matrix of $K_{m,n}$ has 3 distinct
eigenvalues. This completes the proof. \end{proof}

Next, we consider a group of bipartite graphs for which the lower bound
given in Theorem~\ref{thm:uniqueminpath} is tight. This
family is closely related to the ``tadpole graphs'' discussed in
\cite{F} and are of interest since they are {\em parallel paths}
(these graphs are discussed in Section~\ref{sec:n-1}). The exact value
of the maximum multiplicity of parallel paths is known to be $2$ (see \cite{JLS}).

Define $S_{m,n}$ to be the graph consisting of a $4$-cycle on vertices $v_1,u_1,v_2,u_2$ and edges $u_1v_2, v_2u_2, u_2v_1, v_1u_1$, together with a path $P_{m+1}$ starting at vertex $v_1$ and a path $P_{n+1}$ starting at $v_2$, where $P_{m+1}$ and $P_{n+1}$ are disjoint from each other and they intersect the $4$-cycle only on $v_1$ and $v_2$, respectively. Label the vertices on the paths $P_{m+1}$ and $P_{n+1}$ by $u_i$ and $v_i$, alternating the label $u$ and $v$ so that the graph can be considered as a bipartite graph with parts consisting of vertex sets $\{u_i\}$ and $\{v_i\}$. The graph $S_{m,n}$ has $m+n+4$ vertices, and the graph $S_{4,4}$ is given in Figure~\ref{fig:dart}.

\begin{figure}[htbp]
\begin{center}
\begin{tikzpicture}
\draw [fill] 
(-3,0)  circle (2pt) node[above]{$v_4$}
-- (-2.5,0)  circle (2pt) node[above]{$u_4$}
-- (-2,0)  circle (2pt) node[above]{$v_3$}
-- (-1.5,0)  circle (2pt) node[above]{$u_3$}
-- (-1,0)  circle (2pt) node[above]{$v_1$}
-- (-.5,.5)  circle (2pt) node[above]{$u_1$}
-- (0,0)  circle (2pt) node[above]{$v_2$}
-- (.5,0)  circle (2pt) node[above]{$u_5$}
-- (1,0)  circle (2pt) node[above]{$v_5$}
-- (1.5,0)  circle (2pt) node[above]{$u_6$}
-- (2,0)  circle (2pt) node[above]{$v_6$};
\draw [fill] (-.5,-.5)  circle (2pt) node[below]{$u_2$};
\draw (-1,0)--(-.5,-.5)--(0,0);

\end{tikzpicture}
\caption{The graph $S_{4,4}$ and $q(S_{4,4})=6$.}
\label{fig:dart}
\end{center}
\end{figure}
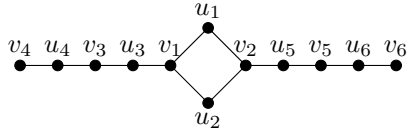

\begin{lemma}
If $m$ and $n$ have the same parity, then $$q(S_{m,n}) = \max\{m,n\}+2.$$
\end{lemma}
\begin{proof} We assume that $m$ and $n$ are both even, the case when they are both
odd is similar. We use the above labeling for $S_{m,n}$, and assume that $m \geq n$. Since there is a unique shortest path
with $m+2$ vertices from the pendant vertex on $P_{m+1}$ to $u_1$, 
by Theorem \ref{thm:uniqueminpath}
we know that $q(S_{m,n}) \geq m+2$. Define an $(m+n+4)/2 \times (m+n+4)/2$ matrix $B=[b_{ij}] \in
\mathcal{B}(S_{m,n})$ with the rows labeled by the vertices $u_i$ and
the columns labeled by the vertices $v_i$. Let $b_{u_2
v_1}=-1$ and  $b_{u_2v_2}=b_{u_1v_1}=b_{u_1v_2}=1$. 
Then, with the proper ordering of the vertices,
$BB^T$ has the form $\left[
\begin{matrix}
X & 0 \\
0 & Y \\
\end{matrix}
\right]$
where $X$ is an $\frac{m+2}{2}\times \frac{m+2}{2}$ tridiagonal matrix
and $Y$ is an $\frac{n+2}{2}\times \frac{n+2}{2}$ tridiagonal
matrix. Using the inverse eigenvalue problem for tridiagonal matrices
\cite{ld} it is possible to find entries for $B$ such that the
eigenvalues for $X$ are distinct and the eigenvalues for $Y$ are a
subset of the eigenvalues of $X$. Thus $q(BB^T) = \frac{m+2}{2}$ and by
Corollary~\ref{cor:bipartite2q}, $q(S_{m,n}) \leq m+2$.  \end{proof}

Since $q(P_n)=n$ and $q(C_n)\approx n/2$, we know that addition of an edge can dramatically decrease the minimum number of
distinct eigenvalues. Here we show that the addition of an edge to a graph can also dramatically increase the minimum number of
distinct eigenvalues. To see this consider the graph $G$ obtained by
adding an edge between vertices $u_1$ and $u_3$ in the graph
$S_{m,m}$ (see Figure~\ref{fig:dart2}). We know that $q(S_{m,m}) =
m+2$, but the new graph $G$ has a unique shortest path that contains
$2m+2$ vertices from a pendant vertex to another pendant vertex. Thus, by 
Theorem \ref{thm:uniqueminpath},
 $q(G)\geq 2m+2$, and
\begin{figure}[htbp]
\begin{center}
\begin{tikzpicture}
\draw [fill] 
(-3,0)  circle (2pt) node[above]{$v_4$}
-- (-2.5,0)  circle (2pt) node[above]{$u_4$}
-- (-2,0)  circle (2pt) node[above]{$v_3$}
-- (-1.5,0)  circle (2pt) node[above]{$u_3$}
-- (-1,0)  circle (2pt) node[below]{$v_1$}
-- (-.5,.5)  circle (2pt) node[above]{$u_1$}
-- (0,0)  circle (2pt) node[above]{$v_2$}
-- (.5,0)  circle (2pt) node[above]{$u_5$}
-- (1,0)  circle (2pt) node[above]{$v_5$}
-- (1.5,0)  circle (2pt) node[above]{$u_6$}
-- (2,0)  circle (2pt) node[above]{$v_6$};
\draw [fill] (-.5,-.5)  circle (2pt) node[below]{$u_2$};
\draw (-1,0)--(-.5,-.5)--(0,0);
\draw (-.5,.5)--(-1.5,0);
\end{tikzpicture}
\caption{The graph $S_{4,4}$ with a single edge added. The parameter $q$ for this graph is at least $10$.}
\label{fig:dart2}
\end{center}
\end{figure}
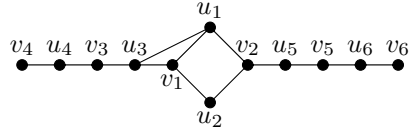
we may conclude that there exist graphs $G$ and an edge $e$ such that
the gap between $q(G)$  and $q(G-e)$ can grow
  arbitrarily large as a function of the number of vertices.

Similarly, if we consider the graph obtained from $S_{m,m}$ by adding a new vertex $w$ and edges 
$\{w,u_1\}$ and $\{w, u_3\}$, then this new graph has a unique shortest path between the pendant vertices that contains $2m+3$ vertices. Hence there exists a family of graphs
$G$ with a vertex $v$ of degree $2$ such that
 the gap between $q(G)$ and $q(G\backslash v)$ can grow arbitrarily large as a function of the number of vertices.

%\section{Graph Products}
%\label{sec:product}

We now switch gears and consider a graph product and a graph operation
in an effort to compute $q$ for more families of graphs. The product that
we consider is the {\em Cartesian product}; if $G$ and $H$ are graphs then
$G \square H$ is the graph on the vertex set $V(G) \times V(H)$ with
$\{g_1,h_1\}$ and $\{g_2,h_2\}$ adjacent if and only if either $g_1=g_2$
and $h_1$ and $h_2$ are adjacent in $H$ or $g_1$ and $g_2$ are
adjacent in $G$ and $h_1 =h_2$.

\begin{theorem}\label{thm:cartesian} 
Let $G$ be a graph on $n$ vertices, then $q(G \square K_2)\leq 2q(G)-2$.
\end{theorem}
\begin{proof} Let $A \in S(G)$ with $q(A) = q(G)=\ell$, and assume $\sigma(A)=\{\lambda_1, \lambda_2, \dots,
\lambda_{\ell}\}$, where $\lambda_1=1, \lambda_2=-1$.  Let $\alpha,\beta$ be nonzero scalars, and consider the matrix
\[
B = \left[ \begin{array}{cc} \alpha A & \beta I \\ \beta I & -\alpha A \end{array}\right]\in S(G \square K_2).
\]
Then,
\[
B^2 = \left[ \begin{array}{cc} \alpha^2 A^2 + \beta^2I & 0 \\ 0 & \alpha^2 A^2+ \beta^2 I \end{array}\right]
\]
and the eigenvalues of $B^2$ are of the form $\alpha^2\lambda_i^2 + \beta^2$, for $i=1,\ldots,\ell$.
If we choose $\alpha^2+\beta^2 = 1$, then two eigenvalues of $B^2$ are equal to $1$.
Since $\pm \lambda$ is an eigenvalue of $B$ whenever $\lambda^2$ is an eigenvalue of $B^2$, this implies that
$q(B) \leq 2(\ell-1) = 2q(G) - 2$.\end{proof}

The following is implied by Lemma~\ref{lem:q=1} and Theorem~\ref{thm:cartesian}. 

\begin{corollary}\label{thm:hypercube}
If $q(G) =2$, then $q(G \square K_2) = 2$.
\end{corollary}

Observe that Corollary \ref{thm:hypercube} verifies that the bound in Theorem~\ref{thm:cartesian} can be tight.

\begin{corollary}\label{ex:hypercube}
If $n\geq 1$ is an integer, then the hypercube, $Q_n$, satisfies $q(Q_n) = 2$. 
\end{corollary}

\begin{proof}
Recall that $Q_n$ can be defined recursively as $Q_n = Q_{n-1} \square K_2$, with $Q_1=K_2$.
Since $q(K_2)=2$, the results follows by application of Corollary \ref{thm:hypercube}.
\end{proof}

Note that the diameter of $Q_n$ is $n$ while $q(Q_n)$ is always $2$,
so for a graph that is not a tree the difference between the diameter
and the minimum number of distinct eigenvalues can be arbitrarily large, as 
a function of the number of vertices.

Next we consider an operation on a graph.
Let $G$ be a graph, then the {\em corona of $G$} is the graph
formed by joining a pendant vertex to each vertex of $G$.

\begin{lemma}\label{lem:1corona}
Let $G$ be a graph and let $G'$ be the corona of $G$, then 
$q(G') \leq 2q(G)$.
\end{lemma}
\begin{proof} Consider the matrix
$
B= \left[ \begin{array}{cc}
A & I \\ I & 0
\end{array} \right]\in S(G')
$ when $A\in S(G)$. Assume that $\lambda$ is an eigenvalue of $B$ with the eigenvector $
\left[ \begin{array}{c}
x \\ y
\end{array} \right]$. Then, 
$Ax + y = \lambda x$ and $x = \lambda y.$
Hence, $\lambda\neq 0$, and $Ay = \frac{\lambda^2-1}{\lambda} y$. Therefore, $\mu = \frac{\lambda^2 -1}{\lambda}$ is an eigenvalue of $A$.
This implies that for each eigenvalue $\mu$ of $A$ there are two real eigenvalues for $\lambda=\frac{\mu\pm\sqrt{\mu^2+4}}{2}$, this completes the proof.\end{proof}

\section{Connected graphs with many distinct eigenvalues}
\label{sec:n-1}

In this section we address the question of ``which connected graphs have the minimum
number of distinct eigenvalues near the number of vertices of the
graph".  In Proposition~\ref{paths}, we observed that $q(G) = |V(G)|$ if and only
if $G$ is a path. In this section we study the connected graphs $G$ with the property that
 $q(G) = |V(G)|-1$.

To begin we apply Theorem~\ref{thm:uniqueminpath} to derive two families of graphs
for which $q$ is one less than the number of vertices. 

\begin{proposition} \label{q:n-1-triangle}
  Let $G$ be the graph with vertices $v_1,v_2,\dots ,v_n$ and edge set
   $E = \{\{v_1,v_2\} ,\{v_2, v_3\},\ldots ,\{v_{n-1},v_n\}, \{v_i, v_{i+2}\}\}$, where
   $i$ is fixed and satisfies $1 \leq i \leq n-2$. Then, $q(G) = |V(G)|-1$.
\end{proposition}

\begin{proposition}\label{Cor:treeswithhighq}
 Let $G$ be the graph with vertices $v_1,v_2,\dots ,v_n$
and edge set
$E = \{\{v_1,v_2\} ,\{v_2, v_3\},\ldots ,\{v_{n-2},v_{n-1}\}, \{v_i, v_{n}\}\}$, where
$i$ is fixed and satisfies $2 \leq i \leq n-2$. 
Then, $q(G) = |V(G)|-1$.
\end{proposition}

Using Proposition~\ref{mrandq} we may deduce that any graph $G$ for which $q(G) = |V(G)|-1$ implies $\M(G) = 2$. However, even more can be said about such graphs.

\begin{theorem}\label{Cor:graphsswithhighq}
If $G$ is a graph that satisfies $q(G)=|V(G)|-1$, then $G$ has the following properties:
\begin{enumerate}
\item $\M(G) = 2$.
\item If $A$ is in $S(G)$ and $A$ has a multiple eigenvalue, then $A$ has exactly one eigenvalue of multiplicity two,
and all remaining eigenvalues are simple. 
\end{enumerate}
\end{theorem}

The next result verifies that the graphs in Proposition~\ref{Cor:treeswithhighq} are the only trees with $q(G)=|V(G)|-1$.

\begin{lemma} \label{treeqn-1} Suppose $T$ is a tree. If $q(T) = |V(T)|-1$, then $T$ consists of a path
$P_{|V(T)|-1}$, along with a pendant vertex adjacent to a non-pendant vertex in this path.
\end{lemma}
\begin{proof} Since $q(T)=|V(T)|-1$, it follows from Proposition~\ref{mrandq}
that $\mr(T) \geq |V(T)|-2$.  Using Theorem~\ref{Cor:graphsswithhighq} (1), $\M(T)=2$, and
hence the vertices of $T$ can be covered by two vertex-disjoint paths (see \cite{JD}).
Therefore, $T$
consists of two induced paths $P_1$ and $P_2$ that cover all of the
vertices of $T$ along with exactly one edge connecting $P_1$ and
$P_2$. Then $T$ has maximum degree equal to three and contains at most
two vertices of degree three. Using Theorem~\ref{Cor:graphsswithhighq} (2), if $q(T)=|V(T)|-1$, then any matrix $A\in S(T)$
realizing an eigenvalue of (maximum) multiplicity two, has all
other eigenvalue being simple. In \cite{JS}, all such trees have been
characterized, for all values of $\M(T)$. In particular, from Theorem
1 in \cite{JS}, we may conclude that the subgraph of $T$ induced by
the vertices of degree at least three must be empty.  Thus, $T$
has exactly one vertex of degree three. Furthermore, deletion of the
vertex of degree three yields at most two components that contain more
than one vertex (see \cite[Thm.  1]{JS}), and hence must be of the
claimed form.   \end{proof}

Characterizing general connected graphs $G$ with the property that $q(G) =$ \\ 
$|V(G)|-1$ appears to be
rather more complicated. By Theorem~\ref{Cor:graphsswithhighq}, we can restrict 
attention to certain graphs with $\M(G) =2$. Fortunately, the graphs with $\M(G) =2$ have been characterized in~\cite{JLS} and they
include the graphs known as graphs of two parallel paths. A
graph $G$ is {\em a graph of two parallel paths} if there exist two disjoint
induced paths (each on at least one vertex) that cover the vertices of $G$ and any
edge between these two paths can be drawn so as not to cross other edges (that is,
there exists a planar embedding of $G$). The graphs $S_{m,n}$ described in Section~\ref{sec:bipartite} are examples of graphs of two
parallel paths that satisfy $q(S_{m,n} ) < |V(S_{m,n})| -1$.
Using \cite{JLS}, our investigation reduces to testing, which graphs $G$ either of two parallel
paths or from the exceptional list given in \cite{JLS}, satisfy $q(G) = |V(G)|-1$. 

We first, consider those graphs identified as exceptional type in \cite[Fig. B1]{JLS}. We let 
$C_5$, $C_5'$ and $C_5''$ denote the graphs pictured in Figure \ref{exp-type}, and refer to them
as base exceptional graphs, from which all other exceptional graphs can be formed by attaching
paths of various lengths to the five vertices in each of $C_5$, $C_5'$ and $C_5''$.

\begin{figure}[h]
\begin{center}
\begin{tabular}{ccc}
\begin{tikzpicture}
\draw [fill] 
-- (-1,0)  circle (2pt) node[above]{$1$}
-- (0,0)  circle (2pt) node[above]{$5$}
-- (1,0)  circle (2pt) node[above]{$4$}
-- (.5,-1)  circle (2pt) node[below]{$3$}
-- (-.5,-1)  circle (2pt) node[below]{$2$};
\draw (-1,0)--(-.5,-1);
\end{tikzpicture}
&
\begin{tikzpicture}
\draw [fill] 
-- (-1,0)  circle (2pt) node[above]{$1$}
-- (0,0)  circle (2pt) node[above]{$5$}
-- (1,0)  circle (2pt) node[above]{$4$}
-- (.5,-1)  circle (2pt) node[below]{$3$}
-- (-.5,-1)  circle (2pt) node[below]{$2$};
\draw (-1,0)--(-.5,-1);
\draw (0,0)--(.5,-1);
\end{tikzpicture}
&
\begin{tikzpicture}
\draw [fill] 
-- (-1,0)  circle (2pt) node[above]{$1$}
-- (0,0)  circle (2pt) node[above]{$5$}
-- (1,0)  circle (2pt) node[above]{$4$}
-- (.5,-1)  circle (2pt) node[below]{$3$}
-- (-.5,-1)  circle (2pt) node[below]{$2$};
\draw (-1,0)--(-.5,-1)--(0,0)--(.5,-1);
\end{tikzpicture}\\
$C_5$ & $C_5'$ & $C_5''$
\end{tabular}
\caption{Base Exceptional Graphs} 
\label{exp-type}
\end{center}
\end{figure}
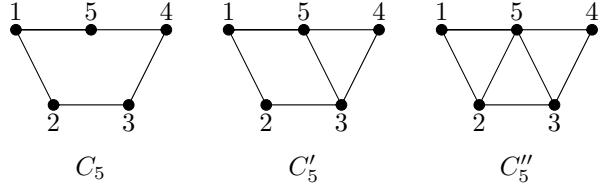

\begin{lemma}\label{coreexp}
Each of the graphs $C_5$, $C_5'$ and $C_5''$ in Figure \ref{exp-type} satisfy
\[ q(C_5) =q(C_5') = q(C_5'') = 3. \]
\end{lemma}

\begin{proof}
We already know that $q(C_5)=3$. For the remaining equalities it is enough to demonstrate 
the existence of a matrix with three distinct eigenvalues, since using Theorem~\ref{thm:uniqueminpath} implies $q(C_5'), q(C_5'')\geq 3$. Consider $C_5''$ first. Let $A \in S(C_5'')$ be of the form
\[ A = \left[ \begin{array}{c|c}
L & b \\ \hline & \\
 b^T & a \end{array} \right], \]
where
\[ L = \left[ \begin{array}{cccc}
1 & -1 & 0 & 0 \\
-1 & 2 & -1 & 0 \\
0 & -1 & 2 & -1 \\
0 & 0 & -1 & 1 \end{array} \right] \] is the Laplacian matrix for a path on four vertices, namely
$\{1,2,3,4\}$. We will determine $a$ and $b$ based on some conditions in what follows. It is not
difficult to check that the eigenvalues of $L$ are $\{0,2, 2\pm \sqrt{2}\}$. The objective here is
to choose $a$ and $b$ so that $A \in S(C_5'')$ and the eigenvalues of $A$ are $\{ 0,0,2,2,\lambda\}$ ($\lambda \neq 0, 2$). This can be accomplished by satisfying the following conditions:
\begin{enumerate}
\item $b = Lu = (L-2I)w$ for some real vectors $u,w$;
\item $a = u^T Lu = w^T (L-2I)w +2$; and
\item $b$ has no zero entries.\end{enumerate}
If the eigenvectors of $L$ associated with $ 2 \pm \sqrt{2}$ are $x_2$ and $x_3$, respectively, 
then we know that $u,w$ must be in the span of $\{x_2, x_3\}$. Hence we can write 
\[ u = \alpha_1 x_2 + \beta_1 x_3,\; {\rm and} \; w = \alpha_2 x_2 + \beta_2 x_3, \]
for some scalars $\alpha_1, \beta_1, \alpha_2, \beta_2$. In this case, (2) can be re-written
as 
\[ (2+ \sqrt{2})\alpha_1^2+ (2-\sqrt{2})\beta_1^2 = \sqrt{2}\alpha_2^2 - \sqrt{2}\beta_1^2 +2,\]
and (1) can be re-written as 
\[ (2+ \sqrt{2})\alpha_1 x_2+ (2-\sqrt{2})\beta_1 x_3 = \sqrt{2}\alpha_2 x_2 - \sqrt{2}\beta_1 x_3.\]
Since $\{x_2,x_3\}$ forms a linearly independent set of vectors, we have
\[ \alpha_2 = \left( \frac{2+\sqrt{2}}{\sqrt{2}}\right) \alpha_1, \; {\rm and} \;
\beta_2 = \left( \frac{\sqrt{2}-2}{\sqrt{2}}\right) \beta_1.\]
Substituting these values back into (2) gives,
\[ (2+ \sqrt{2})\alpha_1^2 = (2-\sqrt{2})\beta_1^2 - \sqrt{2}.\]
Thus choosing $\beta_1$ large enough will suffice in satisfying all of the conditions (1)-(3) 
above. For example, if 
\[\beta_1 = 2 \; {\rm and} \; \alpha_1 = - \sqrt{\frac{(2-\sqrt{2})\beta_1^2 - \sqrt{2}}{2+\sqrt{2}}},\] 
then $A$, as constructed above, will have the desired 
form (that is $A \in S(C_5'')$) and with prescribed eigenvalues
$\{0,0,2,2,18-9\sqrt{2}\}$. (The actual entries of $A$ cannot be easily simplified so we 
have not displayed it here.) Hence  $q(C_5'')=3$. Similar arguments can be applied to the
graph $C_5'$, to conclude that 
$q(C_5')=3$ as well.
\end{proof}

In fact, using the above techniques, and the results obtained thus far we may deduce the following
result.

\begin{theorem}\label{main-atmost5}
For connected graphs $G$ on at most five vertices only the graphs from Propositions
\ref{q:n-1-triangle} and \ref{Cor:treeswithhighq} satisfy $q(G)=|V(G)|-1$.
\end{theorem}

We have poured considerable effort into extending the above fact to larger orders, but
this still has not been resolved. However, we have a strong suspicion that this fact can
be extended. For instance, by Lemma \ref{treeqn-1}, this is true for trees. Furthermore,
using Lemma \ref{coreexp} and considering the result in the previous section on coronas,
we feel strongly that all of the exceptional graphs $G$ listed in \cite[Fig. B1]{JLS} satisfy
$q(G)< |V(G)|-1$.

\section{Possible future directions}

There are many open questions concerning the minimum number
of distinct eigenvalues of a graph. In this section we list some of them that we find interesting and provide some possible directions along these lines.

We have seen that adding an edge or a vertex can dramatically change
the minimum number of distinct eigenvalues of a graph but we suspect that
adding a pendant vertex to a graph could increase the minimum number
of distinct eigenvalues by at most one. The next problem that we plan to work
on is to determine how adding pendant vertices to a graph affects
the minimum number of distinct eigenvalues.

We are also interested in how other graph operations affect the minimum number
of distinct eigenvalues.  For example, can we determine the minimum number of
distinct eigenvalues of a graph that is the vertex sum of two graphs?
Or what is the value of $q(G_1 \join G_2)$ or $q(G \join G \join G)$ in general?
Similarly, does Theorem \ref{join:thm} still hold if $G$ is disconnected?
We formulate the following unresolved idea for the join of two distinct graphs.
  If $G_1$ and $G_2$ are connected graphs and $|q(G_1) -q(G_2)|$ is
  small, then is $q(G_1 \join G_2)=2$?

Another unresolved issue deals with strongly-regular graphs. For
any nonempty strongly-regular graph $G$, it is clear that $2 \leq q(G)\leq3$. Thus a key
 question is
which strongly-regular graphs satisfy $q(G) =2$? The complete bipartite
graph $K_{n,n}$ and $K_n$ are examples of such graphs. By
Corollary~\ref{cor:22adj}, if $G$ is a strongly-regular graph with
parameters $(n,k,a,c)$ (see \cite[Chap. 5]{BR}), where $c$ is the number of mutual neighbors of
any two non-adjancent vertices, and $q(G) = 2$, then $c \geq 2$ (but this is
hardly a strong restriction). However, this restriction on $c$ does verify that the minimum number of distinct eigenvalues for the Petersen graph is three. In addition the 
complete multi-partite graphs of the form $G=K_{n_1, n_1, n_2,n_2, \dots, n_k, n_k}$, where $n_i$ are arbitrary positive numbers, have $q(G)=2$; because $G$ is the join of $K_{n_1,n_2,\dots, n_k}$ with itself; and obviously $G$ is strongly regular.
Also, $q(K_{2,2,2})=2$ as a $6 \times 6$ real symmetric orthogonal matrix can be constructed with $2 \times 2$ zero blocks on the diagonal and Hadamard-like $2 \times 2$ matrices off the diagonal. Observe that $K_{2,2,2}$  is also strongly regular but is not the join of a graph with itself. At present we are still not sure about $q(K_{2,2,\dots,2})$ or $q(K_{3,3,3})$.

Finally, the last outstanding issue is the
characterization of all graphs $G$ for which $q(G) =
|V(G)|-1$. Towards this end, as we eluded to in Section 7, we presented a 
number of ideas and directions towards a general characterization.

\bigskip
{\bf Acknowledgment:} We would like to thank Dr. Francesco Barioli and Dr. Robert Bailey for
a number of interesting discussions related to this topic, and other connections to
certain spectral graph theory problems.

%%%%%%%%%%%%%%%%%%%%%%%%%%%%%%%%%%%%%%%%%%%%%%%%%%%%%%

\end{document}